\documentclass[final]{siamart0516}

\usepackage{amsmath}
\newcommand{\ud}{\mathrm{d}}
\usepackage{bm}
\usepackage{amssymb}
\usepackage{amsfonts}

\usepackage{comment}

\usepackage{cases}

\usepackage{graphicx}

\renewcommand\thesection{\arabic{section}}
\numberwithin{equation}{section}

\usepackage{xcolor}

\newcommand{\KKer}{\mbox{Ker}(\mathcal{K})}
\newcommand{\KKerH}{\mbox{Ker}(\mathcal{K}_h)}
\newcommand{\KKerOrth}{\mbox{Ker}(\mathcal{K})^\perp}

\newcommand{\KKerHOrth}{\mbox{Ker}(\mathcal{K}_h)^\perp}

\renewtheorem{theorem}{Theorem}[section]

\newtheorem{remark}[theorem]{Remark}

\ifpdf
  \DeclareGraphicsExtensions{.eps,.pdf,.png,.jpg}
\else
  \DeclareGraphicsExtensions{.eps}
\fi

\newcommand{\TheTitle}{Guaranteed eigenvalue bounds for the Steklov eigenvalue problem}
\newcommand{\TheAuthors}{C. You, H. Xie and  X. Liu}

\headers{\TheTitle}{\TheAuthors}

\title{{\TheTitle}\thanks{Submitted to the editors DATE.
\funding{The research has been supported by Science Challenge Project (No. TZ2016002), National Natural Science Foundations of China (NSFC 11771434, 91330202, 11371026),
the National Center for Mathematics and Interdisciplinary Science, CAS, for the second author;
by Japan Society for the Promotion of Science, Grand-in-Aid for Young Scientist (B) 26800090 and Grant-in-Aid for Scientific Research (C) 18K03411 for the third author.}}}

\author{
Chun'guang You\thanks{CAEP Software Center for High Performance Numerical Simulation, NO. 6, Huayuan Road, Haidian District, Beijing 100088, China (\email{youchg@lsec.cc.ac.cn}).}
\and
Hehu Xie\thanks{LSEC, ICMSEC, Academy of Mathematics and Systems Science, Chinese Academy of Sciences, NO. 55, Zhongguancun East Road, Beijing 100190, P.R. China, and School of Mathematical Sciences, University of Chinese Academy of Sciences, Beijing, 100049, P.R. China (\email{hhxie@lsec.cc.ac.cn}).}
\and
Xuefeng Liu\thanks{(Corresponding author) Graduate School of Science and Technology, Niigata University, 8050 Ikarashi 2-no-cho,
Nishi-ku, Niigata City, Niigata 950-2181 Japan (\email{xfliu@math.sc.niigata-u.ac.jp}).}
}

\ifpdf
\hypersetup{
  pdftitle={\TheTitle},
  pdfauthor={\TheAuthors}
}
\fi

\begin{document}
%\title{Verified eigenvalue bounds for the Steklov eigenvalue problem}
%-----------------------------------------------------------------------------------------------------

\date{}
\maketitle
%-----------------------------------------------------------------------------------------------------
%-----------------------------------------------------------------------------------------------------
\begin{abstract}
To provide mathematically rigorous eigenvalue bounds for the Steklov eigenvalue problem,
an enhanced version of the eigenvalue estimation algorithm developed by the third author is proposed, 
which removes the requirements of the positive definiteness of bilinear forms in the formulation of eigenvalue problems.
In practical eigenvalue estimation, the Crouzeix--Raviart finite element method (FEM) along with
quantitative error estimation is adopted.
Numerical experiments for eigenvalue problems defined on a square domain and an L-shaped domain are provided to validate the precision of computed eigenvalue bounds.

%-----------------------------------------------------------------------------------------------------
\end{abstract}
%-----------------------------------------------------------------------------------------------------
%-----------------------------------------------------------------------------------------------------

\section{Introduction}
We aim to provide explicit eigenvalue bounds for Steklov-type eigenvalue problems,
such as
\begin{equation} \label{eq:steklov-eig-pro}
-\Delta u + u = 0 \mbox{ in }\Omega, \quad  \frac{\partial{u}}{\partial \vec{n}}=\lambda  u \mbox{ on  } \partial \Omega\:.
\end{equation}
Here,
$\Omega$ is a bounded domain with Lipschitz boundary and ${\vec{n}}$ is the unit outward normal on the boundary $\partial\Omega$.
Such problems have increasing sequences of eigenvalues (see, for example, \cite{Babuska-Osborn-1991}):
$$
0<\lambda_1 \le \lambda_2 \le \ldots\, .
$$
Eigenvalue problems with eigenvalue parameters in the boundary conditions appear in many practical applications.
For example, they can be found when modeling anti-plane shearing in a system of
collinear faults under a slip-dependent friction law \cite{Bucur-Ionescu-2006}, or
the vibration modes of a linear elastic structure
containing an inviscid fluid \cite{Bermudez-Rodriguez-Santamarina-2000}.

It is important to obtain concrete values or exact eigenvalue bounds.
For example, in the error analysis when verifying solutions to nonlinear partial differential equations,
{{explicit values of many error constants are desired \cite{PLUM1992,NAKAO1992,Takayasu2013}}}.
Such constants are often determined by solving differential eigenvalue problems; see \cite{Kikuchi+Liu2007,liu-kikuchi-2010}.
As a concrete example, the constant in the trace theorem is directly related to the Steklov eigenvalue problem.
The trace theorem states that for all $H^{1}$ functions defined on a domain $\Omega$ with Lipschitz boundary,
there exists a constant $C$ that makes the following inequality hold:
$$
\|u\|_{L^2(\partial \Omega)} \le C \|u \|_{H^1(\Omega)} \quad \forall u \in H^1(\Omega)\:.
$$
The constant $C$ here is determined by the first eigenvalue of the eigenvalue problem \cref{eq:steklov-eig-pro}, that is,
$C={1}/{\sqrt{\lambda_1}}$.

The Steklov eigenvalue problem belongs to the class of eigenvalue problems involving self-adjoint differential operators,
such as the Laplacian eigenvalue problem.
It is known that upper eigenvalue bounds can easily be obtained using the Rayleigh--Ritz method with trial functions,
e.g., using polynomial trigonometric functions and finite element methods (FEMs).
However, finding lower eigenvalue bounds remains a difficult problem and has drawn the interest of many researchers.
In the literature, various techniques have been developed for providing lower eigenvalue bounds;
see, for example, the survey in \cite{Liu-Oishi-2013} and the papers cited therein.
Note that most of the existing methods only work for special domains and the computed results cannot be guaranteed to be mathematically correct.

To give guaranteed eigenvalue bounds, the rounding error in the floating-point number computing should also be estimated.
Early work to provide mathematically rigorous eigenvalue bounds can be found in Plum \cite{plum1991bounds}, where the homotopy method is developed, and Nakao, et. al. \cite{nakao1999numerical},
which provides eigenvalue bounds by identifying ranges where eigenvalues can and cannot exist.

In research on FEMs, bounding eigenvalues from two sides is an important topic.
The approximate eigenvalue given by the mass lumping method itself is a lower bound, 
but the approximate eigenvalue can only be shown to be an exact lower bound for special domains with well-constructed meshes \cite{Hu-Huang-Shen-2004}. 
Many non-conforming FEMs also provide lower eigenvalue bounds asymptotically, i.e., when the mesh is fine enough, 
the computed eigenvalues converge to the exact values from below; see the work surveyed in \cite{Yang2010, Luo-Lin-Xie-2012} 
and an efficiency improvement using a multilevel correction scheme \cite{Han-Li-Xie-2015}. 
However, the precondition required for these asymptotic lower bounds, 
that the mesh size be small enough, cannot be verified in solving practical problems.

%Nevertheless, it is difficult to find the bound for a general domain, especially for domains with a reentrant corner, since the
%singularities appear in eigenfunctions make the error analysis argument more difficult; see, for example, \cite{Weinberger-1958}.

Birkhoff, et al. \cite{Birkhoff_etal1966} proposed a method of finding eigenvalue bounds for smooth Sturm--Liouville systems using piecewise-cubic polynomials. 
Inspired by the idea of \cite{Birkhoff_etal1966} and using the hypercircle equation technique along with the linear conforming FEM and the lowest-order 
Raviart--Thomas FEM, in \cite{Liu-Oishi-2013,Liu2011}, Liu and Oishi developed an algorithm to provide guaranteed two-sided bounds for the Laplacian eigenvalue problem, 
which can naturally handle eigenvalue problems over bounded polygonal domains of arbitrary shapes. In \cite{Liu-2015}, Liu extends such an algorithm to create an abstract framework for general self-adjoint differential operators. Carstensen et al. \cite{carstensenGallistl2014, Carstensen2014} also developed explicit eigenvalue bounds for Laplace and biharmonic operators. As presented, these eigenvalue bounds require a so called ``separation condition,'' but in reality, this is not needed, as shown in \cite{Liu-2015}.

To deal with the Steklov eigenvalue problem, the framework of \cite{Liu-2015} must be further extended since its bilinear forms \cref{M-N-def} and \cref{steklov problem} are defined on different domains, i.e., the interior of $\Omega$  and the boundary of $\Omega$. In this paper, we extend the framework to more general variationally formulated eigenvalue problems and successfully obtain lower eigenvalue bounds for the Steklov eigenvalue problem along with the Crouzeix--Raviart FEM. The result in \cref{steklov_explicit_eig_bound} shows that the lower bound for the $i$th eigenvalue is obtained as
\begin{equation}
\label{main-result-summary}
\lambda_i \ge  \frac{\lambda_{h,i}}{1+C_h^2 \lambda_{h,i}}\:.
\end{equation}
Here, $\lambda_{h,i}$ is the $i$th approximate eigenvalue computed using the Crouzeix--Raviart FEM (see \S\ref{section-steklov-preliminary} for details) and
$C_h$ is a quantity for which we have worked hard to provide an explicit value.
The Steklov eigenvalue problem is also considered by Ivana and Tom{\'{a}}{\v{s}} in \cite{Sebestova2014}, 
where the {\it a prior--a posteriori} inequalities and a complementarity technique have been applied to calculate two side eigenvalue bounds.
However, the proposed method needs a priori information about exact eigenvalues, which is usually unknown.
To the best of the authors' knowledge, our paper is the first report on rigorous eigenvalue bounds for Steklov eigenvalue problems.

The explicit lower eigenvalue bound \cref{main-result-summary} has a convergence rate
of $O(h)$ (where $h$ is the mesh size) even for convex domains, which is not optimal compared to the convergence rate of $\lambda_{h,i}$.
In \cite{LiLinXie2013}, it is proved that assuming the eigenfunctions are $H^2$-regular, $\lambda_{h,i}$ itself is
an exact lower bound when the mesh size is {\em small enough} and the convergence rate is $O(h^{2})$. \\

%Since it is difficult to construct high-degree non-conforming FEMs, the only way to improve the precision of eigenvalue bounds based on
%\cref{main-result-summary} is to refine the meshes, which is however not efficient.
%
%~ For example, given a unit square domain, by using the second-order Lagrange FEM over an $8\times 8$ mesh,
%~ we {\em prove} that the first eigenvalue is bounded as follows:
%~ $$
%~ 0.240078143 \le \lambda_1 \le 0.240079091  \:.
%~ $$

The rest of the paper is organized as follows.
In \S 2, we introduce the abstractly formulated eigenvalue problem along with the main theorem,
which provides the lower eigenvalue bounds.
In \S 3, results from the previous section are applied to the Steklov eigenvalue problem to obtain lower eigenvalue bounds,
taking care to give explicit error estimates for the projection operator.
In \S 4, computation results are presented to demonstrate the efficiency of our proposed method for bounding eigenvalues. Finally, in \S 5, we summarize the results of this paper and discuss the issues with the current algorithm.

%=============================================================================================
\section{Variationally formulated eigenvalue problem and lower eigenvalue bounds}
First, we formulate the assumptions for the eigenvalue problem in this paper, which can be regarded as an extension of Liu \cite{Liu-2015}.

\begin{itemize}
\item[(A1)]
$\widetilde{{V}}$ is a Hilbert space with inner product ${M}(\cdot, \cdot)$ and norm $\|\cdot \|_{{M}}$.

\item [(A2)]

${N}(\cdot, \cdot)$ is a symmetric positive semi-definite bilinear form of $\widetilde{V}$ and
the corresponding semi-norm is denoted by $\|\cdot \|_{N}$.

\item [(A3)] $\| \cdot \|_{N}$ is compact with respect to $\|\cdot\|_{M}$, i.e.,
every sequence of $\widetilde{V}$ bounded under $\| \cdot\|_M$ has a subsequence that is Cauchy under $\|\cdot\|_{N}$.

\item [(A4)]
%Let ${V}$ and ${V}_h$ be closed linear subspaces of $\widetilde{V}$. Let ${V}_h$ be finite dimentional.
${V}$ and $V_h$ are closed linear subspaces of $\widetilde{V}$, and ${V}_h$ is finite-dimensional.\\

%The restriction of $\widetilde{M}(\cdot,\cdot)$, $\widetilde{N}(\cdot,\cdot)$ to $V$  are denoted by
%$M(\cdot,\cdot)$, $N(\cdot,\cdot)$, respectively.

\end{itemize}

\begin{remark}
In the case where $N(\cdot,\cdot)$ is a positive definite bilinear form, the assumptions here are the same as the ones in \cite{Liu-2015}.
\end{remark}

\begin{remark}
The assumptions (A1)--(A4) are designed to ease the theoretical analysis. For practical problems, the target eigenvalue problems will be configured in the space $V$ and solved numerically with the introduction of $\widetilde{V}$ and $V_h$; see {\rm \S 3} for the case of the Steklov eigenvalue problem.\\
\end{remark}

Let $\mathcal{K}$ be the operator that maps $f\in V$ to the solution $\mathcal{K}f \in V$ of the variational equation
\begin{equation}
\label{eq:def-K}
M(\mathcal{K}f, v) = N(f, v)  \quad \forall v \in V\:.
\end{equation}
Assumptions (A1)--(A4) then assert that  $\mathcal{K}:V \mapsto V$ is a compact self-adjoint operator. \\

\paragraph{\bf Spectrum of $\mathcal{K}$}

Let $\mbox{Ker}( \mathcal{K} )$ be the kernel space of $\mathcal{K}$. Thus, from the definition of $\mathcal{K}$ in \cref{eq:def-K}, we have
$$
\mbox{Ker}( \mathcal{K} ) = \{v \in V \:|\: N(v,v) =0 \}\:.
$$
%~ This is because $\mathcal{K}f=0$ implies $\|f\|_N=0$ by noticing that $\|f\|_N^2 \le \|\mathcal{K}\|_M \|f\|_M$.
Let $\mbox{Ker}(\mathcal{K})^\perp$ denote the orthogonal complement subspace
of $\mbox{Ker}(\mathcal{K})$ in $V$ with respect to $M(\cdot, \cdot)$.
Let
$$d =\mbox{dim}(\mbox{Ker}( \mathcal{K} )^\perp)\:.$$
Note that $d $ can be $\infty$.
From the theory of compact self-adjoint operators, it is well-known that
$\mathcal{K}$ has the
spectrum $\{\mu_k\}_{k=1}^{d }$, and $0$ if $\mbox{Ker}( \mathcal{K} ) \not = \emptyset$.
Here, the sequence $\{\mu_k\}$ is monotonically decreasing, i.e., $\mu_k \ge \mu_{k+1}$;
each $\mu_k$ is positive
and the number of times it occurs is given by its geometric multiplicity.
Let $\{u_k\}_{k=1}^{d }$ be the orthonormal eigenfunctions
associated with  $\mu_k$'s.
Then $V$ has the following orthonormal decomposition:
$$
V = \mbox{Ker}(\mathcal{K}) \oplus  \mbox{Ker}(\mathcal{K})^\perp, \,\,
\quad
\mbox{Ker}(\mathcal{K})^\perp = \mbox{span} \{u_k\}_{k=1}^{d }\:.
$$

\paragraph{\bf Eigenvalue problem}
Consider the following variational eigenvalue problem\footnote{Another formulation of the eigenvalue problem is as follows:
Find $(\lambda,u)\in \mathbb{R}\times V$, s.t. $\|v\|_N=1$ and
$M(u,v) = \lambda N(u, v) \quad \forall  v\in V$. }:
Find $(\lambda,u)\in \mathbb{R}\times {\mbox{Ker}(\mathcal{K})^\perp}$, s.t.
\begin{equation}\label{abstract problem}
M(u,v) = \lambda N(u, v) \quad \forall  v\in {\mbox{Ker}(\mathcal{K})^\perp}.
\end{equation}
Let $\lambda_k := \mu_k^{-1}$.
The eigenpair of \cref{abstract problem} is given by
$\{\lambda_k, u_k\} $ $(k=1, 2, \ldots, d  )$. \\

\paragraph{\bf Discrete eigenvalue problem}
Analogously, introduce $\mathcal{K}_h:V_h \mapsto V_h$ for $f\in V_h$,
\begin{equation}
\label{K_on_V_h}
M(\mathcal{K}_h f, v_h) = N(f, v_h)  \quad \forall v_h \in V_h\:.
\end{equation}
Then, the kernel space of $\mathcal{K}_h$ is given by
$$
\mbox{Ker}( \mathcal{K}_h ) = \{v_h \in V_h \:|\: N(v_h,v_h) =0 \}\:,
$$
and denote its  ${M}$-orthogonal complement in $V_h$ by $\mbox{Ker}( \mathcal{K}_h )^\perp$.

Now let us consider the following eigenvalue problem:
Find $(\lambda_h, u_h)\in \mathbb{R}\times \KKerHOrth$, s.t.,
\begin{equation}\label{abstract fem}
{M}(u_h,v_h) = \lambda_h {N}( u_h, v_h) \quad \forall  v_h\in \KKerHOrth\: .
\end{equation}
Let $n=\mbox{dim}(\KKerHOrth)$ and $\{(\lambda_{h,k}, u_{h,k})\}_{k=1}^n$ be the eigenpairs of \cref{abstract fem} with
$$
0 < \lambda_{h,1} \leq \lambda_{h,2} \leq \ldots \leq \lambda_{h,n}
$$
and ${M}(u_{h,i}, u_{h,j}) = \delta_{ij}$ ($\delta_{ij}$ is the Kronecker delta).

\begin{remark}
In practice, construction of the kernel space can be avoided by defining the eigenvalue problem over $V_h$ as follows:
Find $(\mu_h, u_h)\in \mathbb{R}\times {V_h}$, s.t.
\begin{equation}
\label{abstract_fem_inverse}
N(u_h, v_h ) = \mu_h M(u_h, v_h) \quad \forall  v_h \in  {V_h}.
\end{equation}
Let $\mu_{h,k}$ $(k=1,\ldots, n)$ be the non-zero eigenvalues of \cref{abstract_fem_inverse}.
The eigenvalues of \cref{abstract fem} are simply the inverses of $\mu_{h,k}$'s, i.e., $\lambda_{h,k} = \mu_{h,k}^{-1}$ $(k=1,\ldots, n)$.
Moreover, The eigenvalues of \cref{abstract_fem_inverse} can be rigorously calculated by solving a matrix eigenvalue problem with verified numerical methods.\\

\end{remark}

Denote the Rayleigh quotient over $\widetilde{{V}}$ by $R(\cdot)$ as follows:
for any ${v} \in\widetilde{{V}}$, $\| {v}\|_N \neq 0$,
\begin{equation}
R({v}):=
\frac{M(v,v)}{N(v, v)}.
\end{equation}
The stationary values and points of $R(\cdot)$ over ${V}$ and ${V}_h$ thus
correspond to the eigenpairs of the eigenvalue problems (\ref{abstract problem})
and (\ref{abstract fem}), respectively.
For all $\lambda_k$ and $\lambda_{h,k}$, the min-max principle asserts that
\begin{equation}\label{eq:min-max}
\lambda_k     =\min_{S_k}\max_{\substack{v\in S_k \\ \|v\|_N \neq 0}}R(v)
              = \max_{\substack{v\in E_k \\ \|v\|_N \neq 0}}R(v), \,\,
\lambda_{h,k} =\min_{S_{h,k}}\max_{\substack{v_h\in S_{h,k}\\ \|v_h\|_N \neq 0}}R(v_h)
              = \max_{\substack{v_h\in E_{h,k} \\ \|v_h\|_N \neq 0}}R(v_h),
\end{equation}
where $S_k$ (resp. $S_{h,k}$) is a $k$-dimensional subspace of $\KKerOrth$ (resp. $\KKerHOrth$ )
and $E_k$ (resp. $E_{h,k}$) is the space spanned by the eigenfunctions $\{{u}_i\}_{i=1}^{k}$ (resp. $\{{u}_{h,i}\}_{i=1}^{k}$).\\

Let $P_h:\widetilde{{V}}\mapsto{V}_h$
be the projection with respect to ${M}(\cdot,\cdot)$: given
${u}\in\widetilde{{V}}$, $P_h u \in V_h $ satisfies
\begin{equation}
{M}({u}-P_h {u}, v_h) = 0 \quad \forall v_h\in{V}_h.
\end{equation}

Next, we present a theorem that provides the lower eigenvalue bounds.

\begin{theorem}[Lower eigenvalue bounds]\label{framework_theorem}
Suppose that there exists a positive constant $C_h$ such that
\begin{equation}\label{framework_projection_estimation}
\|u-P_h u\|_{N} \leq C_h \|u-P_h u\|_{M} \quad \forall  u\in {V}\:.
\end{equation}
Let $\lambda_k$ and $\lambda_{h,k}$ be as defined in \cref{abstract problem} and \cref{abstract fem}.
Lower eigenvalue bounds are then given by
\begin{equation}\label{framework_lower_bounds}
\lambda_k \ge \frac{\lambda_{h,k}}{1+C_h^2\lambda_{h,k}}, ~~ k=1,2,\ldots,\min(n,d ).
\end{equation}
\end{theorem}

\begin{proof}
Since $\|\cdot\|_N$ is compact in $\widetilde{{V}}$ with respect to $\|\cdot\|_M$, resulting from the argument of compactness (see \S 8 of \cite{Babuska-Osborn-1989}), there exists $(0<)\overline{\lambda}_1 \le \overline{\lambda}_2 \le \ldots $ such that
\begin{equation}
\label{min--max-max-min}
\overline{\lambda}_k =
\min_{S^k \subset \widetilde{{V}}} \max_{v \in S^k}
R(v)=
\max_{W\subset \widetilde{{V}}, \mbox{dim} ( W) \le k-1} \quad \min_{ v\in W^\perp }
R(v)
\:,
\end{equation}
where $S^k$ denotes any $k$-dimensional subspace of $\widetilde{{V}}$;
$ W^{\perp}$ denotes the orthogonal complement of $W$ in $\widetilde{{V}}$ respect to $M(\cdot, \cdot)$.
Since $V\subset \widetilde{{V}}$, we have $\lambda_k \ge \overline{\lambda}_k$ due to the min--max principle.
Further, by choosing $W$ in (\ref{min--max-max-min}) as $E_{h,k-1}:=\mbox{span}\{u_{h,1},\ldots, u_{h,k-1}\}$,
a lower bound for $\lambda_k$ is obtained:
\begin{equation}
\label{inf-low-bound}
\lambda_k \ge \overline{\lambda}_k \ge \min_{ v \in E_{h,k-1}^{\perp} } R(v)\:.
\end{equation}

Let $E_{h,k-1}^{\perp,h}$ denote the orthogonal complement of $E_{h,k-1} $ in $\KKerHOrth$ with respect to $M(\cdot,\cdot)$.
Hence,
$V_h=E_{h,k-1}\oplus E_{h,k-1}^{\perp,h} \oplus \mbox{Ker}(\mathcal{K}_h) $.
Then $\widetilde{{V}}$ can be decomposed by:
$$
\widetilde{{V}} =V_h \oplus {V_h}^{\perp} = E_{h,k-1} \oplus E_{h,k-1}^{\perp,h}  \oplus \mbox{Ker}(\mathcal{K}_h)  \oplus {V_h}^{\perp}\:.
$$
Moreover, we have $E_{h,k-1}^{\perp}=E_{h,k-1}^{\perp,h}  \oplus \mbox{Ker}(\mathcal{K}_h)  \oplus {V_h}^{\perp}$.
For any $v\in E_{h,k-1}^{\perp}$, we have
$$v=P_h v + (I-P_h)v,\quad P_h v \in E_{h,k-1}^{\perp,h} \oplus \mbox{Ker}(\mathcal{K}_h) ,\quad
(I-P_h)v \in {V_h}^{\perp}\:.
$$
Further, decompose $P_h v$ by $P_h v=v^\perp + v_0 $, where $v^\perp\in E_{h,k-1}^{\perp,h}$, $v_0 \in  \mbox{Ker}(\mathcal{K}_h) $.
Therefore, we have
$\|v^\perp\|_N \le \lambda_{h,k}^{-1/2} \|v^\perp \|_M$ by noticing that
$$
\lambda_{h,k} = \min_{ v \in E_{h,k-1}^{\perp,h} } R(v)\:.
$$
Since $\|P_hv\|_N = \|v^\perp\|_N$, $\|v^\perp\|_M \le \|P_h v\|_M$,
from condition \cref{framework_projection_estimation}, we have
$$
\|v\|_N \le \|P_hv\|_N + \|v-P_hv\|_N \le \lambda_{h,k} ^{-1/2} \|P_hv \|_M + C_h \| v-P_hv\|_M \:.
$$
which leads to
$$
\|v\|^2_N \le \left( \lambda_{h,k}^{-1} + C_h^2 \right) ( \|P_hv \|^2_M + \|v- P_hv \|^2_M)
= \left( \lambda_{h,k}^{-1} + C_h^2 \right) \| v \|^2_M\:.
$$
Hence, we obtain
$$
R(v) \ge
\lambda_{h,k} /\left(1+ C_h^2 \lambda_{h,k} \right) \quad
\mbox{ for any } v \in E_{h,k-1}^{\perp}
\:.
$$
Using (\ref{inf-low-bound}), we can draw the conclusion in \cref{framework_theorem}.
\end{proof}

\vskip 0.5cm

\begin{remark}
\cref{framework_theorem} provides the same lower bound as in \cite{Liu-2015} if
the function space $V$ in \cite{Liu-2015} is taken as $\mbox{Ker}(\mathcal{K})^\perp$ here.
However, one cannot give the proof of \cref{framework_theorem} by just simply
replacing $\widetilde{{V}}$ with $\mbox{Ker}(\mathcal{K})^\perp$ in the proof of \cite{Liu-2015}.
This is because that generally,
\begin{equation}
P_h \left(\KKerOrth \right) \not= \KKerHOrth, \quad 
P_h \left(\KKer \right) \not= \KKerH\:.
\end{equation}
A concrete example is the Steklov eigenvalue problem to be discussed in next section; see Remark \ref{counter-example}.
\quad

\end{remark}

\vskip 0.5cm

\section{The Steklov eigenvalue problem}\label{section-steklov}

In the rest of this paper, we consider eigenvalue bounds for the Steklov eigenvalue problem \cref{eq:steklov-eig-pro},
where $\Omega$ is taken as an $\mathbb{R}^{2}$ domain.
As a remark, the method to be introduced here can be applied to the one with mixed boundary conditions
and there is essentially no difficulty to deal with more general Steklov-type eigenvalue problems.

\subsection{Preliminaries}\label{section-steklov-preliminary}
The analysis is undertaken within the framework of Sobolev spaces.
Let $\Omega$ be a connected bounded domain in $\mathbb{R}^{d}$ ($d=1,2$).
The $L^2(\Omega)$ function space is the set of real square integrable functions over $\Omega$,
for which the inner product is denoted by $(\cdot, \cdot)_{\Omega}$.
We shall use the standard notation for the Sobolev spaces $W^{k,p}(\Omega)$ and their
associated norms $\|\cdot\|_{k,p,\Omega}$ and seminorms $|\cdot|_{k,p,\Omega}$
(see, e.g., Chapter 1 of \cite{Brenner+Scott2002} and Chapter 1 of \cite{Ciarlet2002}).
For $p=2$, we define
$H^k(\Omega)=W^{k,2}(\Omega)$, 
$\|\cdot\|_{k,\Omega} = \|\cdot\|_{k,2,\Omega}$
and
$|\cdot|_{k,\Omega} = |\cdot|_{k,2,\Omega}$.\\

%$H_0^1(\Omega)=\{v\in H^1(\Omega):\ v|_{\partial\Omega}=0\}$,
%where $v|_{\partial\Omega}=0$ is in the sense of trace,
%$\|\cdot\|_{s,\Omega}=\|\cdot\|_{s,2,\Omega}$.

%
%
% and
%then apply \cref{framework_theorem} to obtain lower bounds for Steklov eigenvalue problem.
%

To bound the eigenvalues of \cref{eq:steklov-eig-pro} by applying the results in \S2, we take $V=H^1(\Omega)$ and define
%
%Introduce the inner products $M(\cdot,\cdot)$, $N(\cdot,\cdot)$ for $V$, $W$, respectively.
\begin{equation}
\label{M-N-def}
{M(u,v)} := \displaystyle \int_{\Omega}\nabla u \cdot \nabla v + uv\,\ud\Omega,\quad
{N(u,v)} := \displaystyle{ \int_{\partial\Omega}(\gamma u) \: (\gamma v) \,\ud s \quad \forall  u,v \in {V}},
\end{equation}
where $\gamma:\,\,H^1(\Omega)\mapsto L^2(\partial\Omega)$ is the trace operator.
Under the current domain boundary assumption, $\gamma$ is a compact operator; see, e.g., \cite{Demengel-2012}.
%
%where
%Define $\|\cdot\|_M$ by $\|u\|_M = \sqrt{M(u,u)}$, $\forall u\in {V}$
%and note that $\|u\|_M = \|u\|_{1,\Omega}$; define $\|\cdot\|_N$ by $\|u\|_N = \sqrt{N(u,u)}$,
%$\forall u\in{W}$.

Let us restrict the domain $\Omega$ of \cref{eq:steklov-eig-pro} to be a bounded polygonal domain in $\mathbb{R}^{2}$, and select the finite-dimensional spaces ${V}_h$ as
finite element spaces.
Let $K^h$ be a triangular subdivision of $\Omega$,
$K_b^h$ be the set of elements of $K^h$ having an edge on $\partial \Omega$, $\varepsilon^h$ be the set of all edges of $K^h$, and
$\varepsilon^h_b$ be the set of all boundary edges of $K^h$.
Given an element $K \in K^h$, $h_K$ denotes the length of the longest edge of $K$.
In addition, to make the {\it a priori} error estimate concise (see the proof of \cref{trace theorem corollary 2}),
we further require that all elements $K$ of $K^{h}$ have at most one edge on the boundary
of the domain.

Now, let us introduce the Crouzeix--Raviart finite element space ${V}_h$ on $K^h$:
\begin{equation}
\label{fem space}
\begin{split}
{V}_h := & \{  v \:|\: v \mbox{ is a piecewise-linear function on $K^h$ and continuous at }\\
 & \mspace{28mu} \mbox{the mid-points of  interior edges}\}.
\end{split}
\end{equation}
Since $V_h \not\subset H^1(\Omega)$, we introduce the discrete gradient operator $\nabla_h$, which takes the gradient element-wise for $v_h\in V_h$.
The seminorm $\| \nabla_h v_h \|_{(L^2(\Omega))^2}$ is still denoted by $|v_h|_{1,\Omega}$.

The restriction of the trace operator $\gamma$ to $V_h$, denoted by $\gamma_h$, is well-defined for $v_h \in {V}_h$,
if we regard $\gamma$ as an element-wise operator on the boundary elements.

The extension of $M$ to ${V}_h$ is defined by
\begin{equation}
M(u_h,v_h) := \sum_{K\in K^h}\int_{K}(\nabla u_h \cdot \nabla v_h + u_hv_h)\,\ud K,
                          \ \quad \forall  u_h, v_h\in{V}_h\:.
\end{equation}
Let $\widetilde{V}:=V+V_h$. The above settings thus satisfy the assumptions (A1)--(A4).
The kernel spaces $\KKer$ and $\KKerH$ are determined by the trace operator $\gamma$ as follows:
\begin{eqnarray}
\KKer  &=& \mbox{Ker}(\gamma):= \{v \in H^1(\Omega), v=0 \mbox{ on } \partial \Omega \} \\
\KKerH &=& \mbox{Ker}(\gamma_h):= \{v_h \in V_h, v_h =0 \mbox{ on } \partial \Omega \}.
\end{eqnarray}

With the above definitions in place, we can follow \S2 and define the variational
eigenvalue problem
for \cref{eq:steklov-eig-pro} and the discrete problem in finite element space.

\paragraph*{Variational form of the Steklov eigenvalue problem}
Find $(\lambda,u)\in \mathbb{R}\times \mbox{Ker}(\gamma)^\perp$, s.t.,
\begin{equation}\label{steklov problem}
M(u,v)=\lambda N( u, v), \,\ \quad \forall  v\in \mbox{Ker}(\gamma)^\perp.
\end{equation}

\paragraph*{Eigenvalue problem in finite element space}
Find $(\lambda_h,u_h)\in \mathbb{R}\times \mbox{Ker} (\gamma_h)^\perp$,  s.t.,
\begin{equation}\label{steklov fem}
M(u_h,v_h)=\lambda_h N( u_h, v_h), \,\ \quad \forall  v_h\in \mbox{Ker} (\gamma_h)^\perp.
\end{equation}
We use the same notation for the eigenpairs
of \cref{steklov problem} and \cref{steklov fem} as in \S2. \\

%It is easy to see \cref{steklov fem} has $n$ non-zero eigenvalues, where $n=\mbox{dim}({W}_h)$.

Let $P_h:{V}+{V}_h \mapsto {V}_h$ be the projection operator with respect to
the inner product $M(\cdot,\cdot)$. For $u\in {V}+{V}_h$, $P_h u$ satisfies
\begin{equation}
M(u-P_h u, v_h) = 0 \quad \forall v_h\in {V}_h.
\end{equation}
In \S\ref{sec:c_h_est_interpolation} and \S\ref{sec:c_h_est_projection}, we show how to
obtain the following explicit bound for the projection error constant $C_h$ required by \cref{framework_theorem}:
\begin{equation}
\label{main_result_ch}
\|u-P_h u\|_{L^{2}(\partial \Omega)} \leq C_h \|u-P_h u\|_{1,\Omega} \quad \forall u \in {V}\:.
\end{equation}

In preparation for evaluating $C_h$, let us introduce the Crouzeix--Raviart interpolation operator $\Pi_h: {V}\mapsto {V}_h$, which is defined element-wise.
For any given $K$, whose edges are denoted by $e_1,e_2$ and $e_3$,
$(\Pi_h u)|_K$ is a linear polynomial satisfying
\begin{equation}
\int_{e_i} (\Pi_h u)|_K \,\ud s = \int_{e_i}  u \,\ud s = 0, \,\,i=1,2,3\:.
\end{equation}

\paragraph{Orthogonality of interpolation $\Pi_{h}$ } The interpolation operator $\Pi_h$ has an important orthogonality property:
\begin{equation}
\label{orthogonality_interpolation}
(\nabla_h (\Pi_h u - u), \nabla_h v_h )_\Omega =0 \quad \forall v_h \mbox{ in } V_h\:.
\end{equation}
To prove the above equation, it is sufficient to show that the following equation holds for each element $K$ of $K^h$,
$$
\int_{K} \nabla (\Pi_h u -u) \cdot \nabla v_h \mbox{d} K =
\int_{\partial K} (\Pi_h u -u) \: \nabla  v_h \! \cdot \! \vec{n} \ud s - \int_K  (\Pi_h u -u) \Delta v_h \ud K =0\:.
$$
The last equality holds because of the definition of $\Pi_hu$ and the fact that $v_h|_K$ is a linear function.

\begin{remark}
\label{counter-example}
 The projection operator $P_h(=\Pi_h)$ does not mapping $\KKer$ to $\KKerH$.
A counterexample is to consider the triangulation of the unit square domain with two triangle elements; see Fig. \ref{fig-mesh-two-elements}.
Let $K$ be the element with vertices $(0,0)$, $(0,1)$ and $(1,0)$. Denote the only interior edge by $e$.
Suppose $u=0 $ on $\partial \Omega$ and $\int_e u~ \ud s=1$. Then, $(\Pi_h u)|_K =\sqrt{2}(x+y -1/2)$ and $\Pi_hu\not \in \KKerH$.
\begin{figure}[ht]
\begin{center}
\includegraphics[height=1.5in]{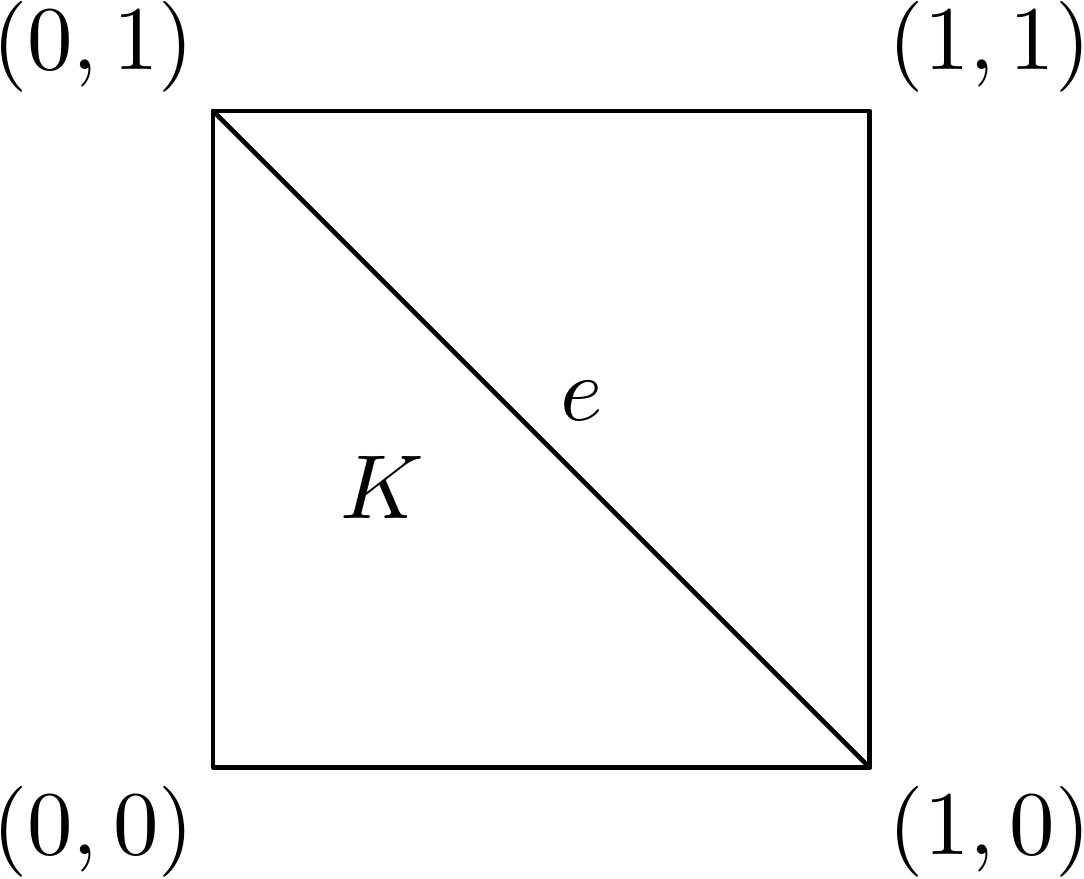}
\end{center}
\caption{Special triangulation of unit square domain \label{fig-mesh-two-elements}}
\end{figure}
\end{remark}

In next subsection, we focus on error estimation for $\Pi_h$, which helps to obtain
the explicit bound \cref{main_result_ch} for $C_h$.

\subsection{Error estimate for interpolation operator}
\label{sec:c_h_est_interpolation}

First, we quote a result of the restriction of $\Pi_{h}$ to an element $K$, which is still denoted by $\Pi_{h}$.

\begin{lemma}[Liu \cite{Liu-2015}]
For any triangle element $K$, whose longest edge length is denoted by $h_K$,
we have
\begin{equation}\label{liu result}
\|u-\Pi_h u\|_{0,K} \leq 0.1893 h_K |u-\Pi_h u|_{1,K} \quad \forall u \in H^1(K)\:.
\end{equation}
\end{lemma}

Next, we estimate $\|u-\Pi_h u\|_{N}$ using information on the boundary elements.

Consider the triangle $K$ (see \cref{Triangle_Figure}),
whose nodes are denoted by $\mathbf{P}_1$, $\mathbf{P}_2$, and $\mathbf{P}_3$. The edge $\mathbf{P}_1\mathbf{P}_2$ is denoted by $e$.
Define the height of triangle $K$ respect to edge $e$ by $H_K$. Thus,
\begin{equation}
\label{two-para-K}
H_K = 2|K|/|e|\:.
\end{equation}

\iffalse
\begin{figure}[h]
\begin{center}
\begin{picture}(0,80)(0,106)
\put(-170,-200){\includegraphics[width=11cm]{graph/tri_trace.ps}}
\end{picture}
\end{center}
\caption{Parameterization of the triangle $K$}
\label{Triangle_Figure}
\end{figure}
\fi

\begin{figure}[h]
\begin{center}
\includegraphics[width=6cm]{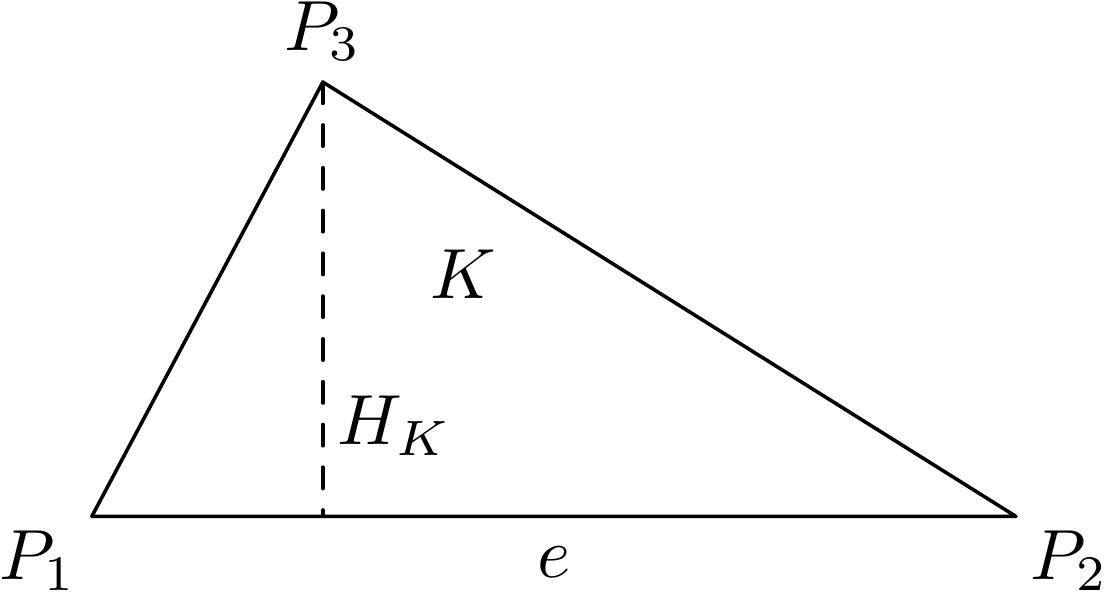}
\end{center}
\caption{Parametrization of the triangle $K$}
\label{Triangle_Figure}
\end{figure}

\begin{theorem}[Interpolation error estimate]\label{trace theorem 2}
For a given element $K$, configured as in \cref{Triangle_Figure},
the following error estimate holds for any $u\in H^1(K)$:
\begin{equation}\label{trace formula 2}
\|u-\Pi_h u\|_{0,e} \leq 0.6711 \frac{h_K}{\sqrt{H_K}}|u-\Pi_h u|_{1,K}.
\end{equation}
\end{theorem}

\begin{proof}
For any $w \in H^{1}(K)$, the Green theorem leads to
$$
\int_K ((x,y)- \mathbf{P_3} ) \cdot \nabla (w^2) \mbox{d}K =
\int_{\partial K} ((x,y)- \mathbf{P_3} )\cdot {\bm{n}} w^2 \mbox{d}s - \int_K 2 w^2 \mbox{d} K
$$
For the term $((x,y)- \mathbf{P_3} )\cdot {\bm{n}}$, we have
\begin{equation}
((x,y)- \mathbf{P_3} )\cdot {\bm{n}} =
\quad
\left\{
\begin{array}{ll}
0, & \mbox{ on }  \mathbf{P_1}\mathbf{P_3}, ~ \mathbf{P_2}\mathbf{P_3}\:, \\
2 {|K|}/{|e|} & \mbox{ on } e\:.
\end{array}
\right.
\end{equation}
Thus,
\begin{align*}
2\frac{|K|}{|e|}\int_e w^2 \ud s & = \int_K 2 w^2 \mbox{d} K + \int_K ((x,y)- \mathbf{P_3} ) \cdot \nabla (w^2) \mbox{d} K\\
& \le  \int_K 2 w^2 \mbox{d} K + 2h_K \int_K |w| |\nabla w| \mbox{d} K \\
& \le 2\|w\|_{0,K}^2 + 2h_K \|w\|_{0,K} \|\nabla w\|_{0,K}\:.
\end{align*}

Taking $w=u-\Pi_h u$ and applying of estimate \cref{liu result}, we have,
\begin{equation*}
 \|w\|_{0,e}  \le \sqrt{ 0.1893^2+0.1893 }\sqrt{ \frac{|e|}{|K|}} ~ h_K ~ \|\nabla w\|_{0,K}
 \le 0.6711 \frac{h_K}{\sqrt{H_K}} ~  \|\nabla w\|_{0,K} \:. 
\end{equation*}
The above inequality gives the desired result.
\end{proof}

Next, let us apply \cref{trace theorem 2} to show the result related to trace theorem.

\begin{corollary}\label{trace theorem corollary 2}
Given $u\in H^1(\Omega)$,
the following error estimate holds:
\begin{equation}\label{error estimation 2}
\|u-\Pi_h u\|_N \leq {0.6711}
\Big(\max_{K\in K_b^h}\frac{h_K}{\sqrt{H_K}}\Big) |u-\Pi_h u|_{1,\Omega}.
\end{equation}
Here, $h_{K}$ is the length of the longest edge of $K$; $H_{K}$ is the height of $K$ defined in \cref{two-para-K}, 
where edge $\mathbf{P}_1\mathbf{P}_2$ is aligned on the boundary of the domain.

\end{corollary}

\begin{proof}
The conclusion follows straightforwardly from \cref{trace theorem 2} by noticing that
\begin{eqnarray*}
\|u-\Pi_h u\|^2_N &=& \sum_{e\in \varepsilon_b^h} \|u-\Pi_h u\|^2_{0,e}
\leq {0.6711}^2
      \max_{K\in K_b^h}\frac{h_K^2}{H_K} \sum_{K\in K_b^h}|u-\Pi_h u|_{1,K}^2 \:.
\end{eqnarray*}
Owing to the assumption that all elements $K$ of ${K}^{h}$ have at most one edge on the boundary, the term $|u-\Pi_h u|_{1,K}$ in the above inequality only needs to be counted at most once.
\end{proof}

\begin{remark}\label{remark:trace-constant}
Numerial computations indicate that, when the height of triangle is fixed, the constant $C$ in the estimate 
$\|u-\Pi_h u \|_{0,e} \le C \|\nabla (u-\Pi_h u)\|_{0,K}$ decreases to zero with rate $C=O(|e|^{1/2})$ as 
$|e|$ tends to $0$.
However, this behavior of the constant $C$ cannot be deduced from \cref{trace theorem 2}.
Below is a sketch of the proof for this property.

Define constants $C(K)$ and $C_e(K)$ by
$$
C(K) =\sup_{v\in H^1(K)} \frac{\|u-\Pi_h u\|_{0,e}}{\|\nabla (u-\Pi_h u)\|_{0,K}},\quad
C_e(K) =\sup_{v\in H^1(K), \, \int_e v \ud s=0} \frac{\|u\|_{0,e}}{\|\nabla u\|_{0,K}}\:.
$$
Then, it is easy to see $C(K)\le C_e(K)$. For the purpose of simplicity in the argument, 
assume $K$ to be an acute triangle. 
Let $\widehat{K}$ be a reference triangle with the length of base of being unit.
Suppose $K$ is obtained by scaling $\widehat{K}$ by $h$ ($h=|e|$) along $x$ direction. 
Let $\widetilde{K}$ be the scaled $\widehat{K}$ by $h$ in both $x$ and $y$ directions; see \cref{fig:three_triangles}. 
Noticing that for any $u\in H^1({K})$, $u|_{\widetilde{K}} \in H^1(\widetilde{K})$, 
$$
\|\nabla (u|_{\widetilde{K}})  \|_{0,\widetilde{K}} \le \|\nabla u\|_{0,{K}}\:.
$$
Thus, one can easily show that
\begin{equation}
\label{trace_constants}
 C_e(K)\le C_e(\widetilde{K}),\quad C_e(\widetilde{K}) = \sqrt{|e|} C_e(\widehat{K})\:.
\end{equation}
Hence,
$$
 C_e(K)\le \sqrt{|e|} C_e(\widehat{K})=O(|e|^{1/2})\:.
$$
For $K$ being an obtuse triangle, the relation $\widetilde{K} \subset {K}$ does not hold any more and transformation of triangles is needed. The argument in this case is 
little complicated and omitted here.
\begin{figure}[!ht]
\begin{center}
\includegraphics[width=3.5cm,angle=0]{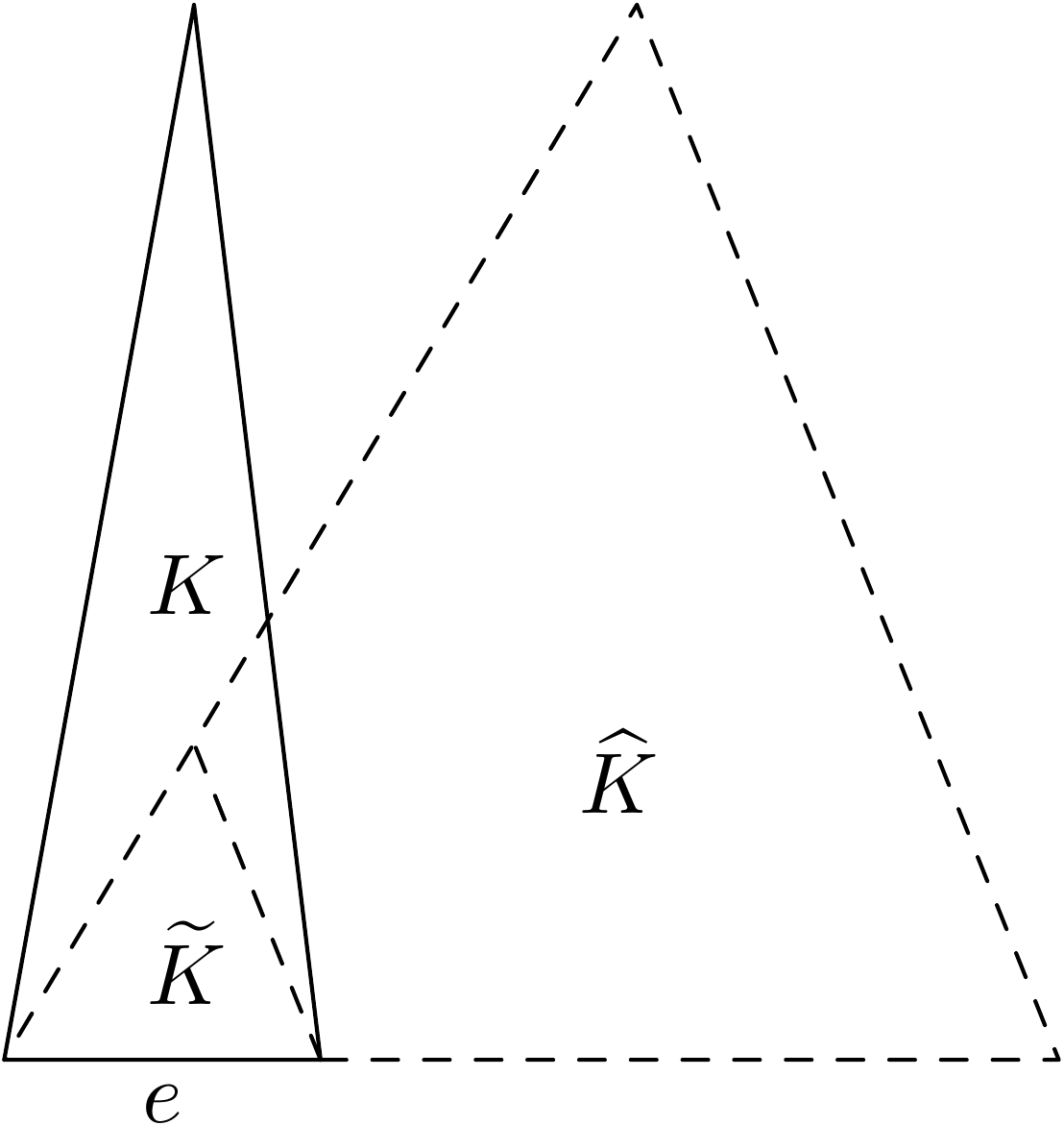}
\caption{Scaling of reference triangle }
\label{fig:three_triangles}
\end{center}
\end{figure}

\end{remark}

\subsection{Error estimate for projection $P_h$}
\label{sec:c_h_est_projection}

In this subsection, we give an explicit bound of $C_h$ required in the error estimate of $P_{h}$ in \cref{main_result_ch} .

First, we estimate the mapping $\mathcal{K}_h$ defined in \cref{K_on_V_h}. That is,
for $\phi_h \in V_h$, $\mathcal{K}_h \phi_h \in V_h$ satisfies
\begin{equation}\label{auxiliary problem}
M(\mathcal{K}_h \phi_h, v_h) = N(\phi_h, v_h) \quad \forall v_h\in {V}_h.
\end{equation}

\begin{lemma}
For all $\phi_h\in {W}_h$, we have
\begin{equation}\label{Inequality_Dual}
\|\mathcal{K}_h \phi_h\|_M \leq \frac{1}{\sqrt{\lambda_{h,1}}} \|\phi_h\|_N,
\end{equation}
where $\lambda_{h,1}$ is the smallest eigenvalue for the discrete Steklov eigenvalue problem \cref{steklov fem}.
\end{lemma}

\begin{proof}
From the definition of $\mathcal{K}_h$ in \cref{auxiliary problem}, by selecting $v_h := \mathcal{K}_h \phi_h$, we obtain
\begin{equation}
\label{K_h_est_1}
\|\mathcal{K}_h \phi_h \|_M^2 = N(\phi_h, \mathcal{K}_h \phi_h) \leq \|\phi_h\|_N \| \mathcal{K}_h \phi_h\|_N.
\end{equation}
From the definition of $\lambda_{h,1}$ in \cref{steklov fem} and the min-max principle, we also have
\begin{equation}
\label{K_h_est_2}
\lambda_{h,1}
%= \min_{\substack{v_h\in {V}_h \\ \gamma v_h \neq 0 }}\frac{\|v_h\|_M^2}{\|v_h\|_N^2}
\leq \frac{\|\mathcal{K}_h \phi_h \|_M^2}{\| \mathcal{K}_h \phi_h\|_N^2}\:,
\mbox{ which implies }
\| \mathcal{K}_h \phi_h\|_N \le \frac{1}{\sqrt{ \lambda_{h,1} }} \|\mathcal{K}_h \phi_h \|_M 　\:.
\end{equation}

We can now complete the proof using \cref{K_h_est_1} and \cref{K_h_est_2}.
\end{proof}

\vskip 0.5cm

The following lemma gives an estimate of the difference between the interpolation and projection operators.
\begin{lemma}
    For all $u \in  H^{1}(\Omega)$, we have
\begin{equation}\label{error estimation 1}
\|\Pi_h u - P_h u \|_N \leq C_1 |u - \Pi_h u|_{1,\Omega},
\end{equation}
where
$$
C_1 := \frac{0.1893}{\sqrt{\lambda_{h,1}}}\max_{K\in K^h} h_K.
$$
\end{lemma}

\begin{proof}
Take $v_h := \Pi_h u - P_h u$ and let $\psi_h:= \mathcal{K}_h \cdot  v_h \in V_h$. Then
\begin{eqnarray*}
\| v_h\|_N^2 &=& N(v_h, v_h)
=M(\psi_h,  v_h)\\
&=& M(\psi_h, \Pi_h u - u + u - P_h u) = M(\psi_h, \Pi_h u-u)\:.
\end{eqnarray*}
Noting the orthogonality of $\Pi_h u$, as shown in \cref{orthogonality_interpolation}, we have
\begin{eqnarray*}
\|v_h\|_N^2 &=& (\psi_h, \Pi_h u-u)_{\Omega}  \leq \|\psi_h\|_{0,\Omega} \|\Pi_h u-u\|_{0,\Omega} \\
&\leq& \|\psi_h\|_{M}  ~  \|\Pi_h u-u\|_{0,\Omega} \leq \frac{1}{\sqrt{\lambda_{h,1}}} \| v_h\|_N ~ \|\Pi_h u-u\|_{0,\Omega}.
\end{eqnarray*}
As a consequence,
$$
\|\Pi_h u - P_h u\|_N \leq \frac{1}{\sqrt{\lambda_{h,1}}}\|\Pi_h u-u\|_{0,\Omega}.
$$
By \cref{liu result}, we have
$$
\|\Pi_h u-u\|_{0,\Omega} \leq 0.1893 \max_{K\in K^h} h_K |\Pi_h u-u|_{1,\Omega}.
$$
Followed by the above two estimates, we obtain
$$
\|\Pi_h u - P_h u\|_N  \leq \frac{0.1893}{\sqrt{\lambda_{h,1}}}\max_{K\in K^h} h_K |u - \Pi_h u|_{1,\Omega}\:.
$$
From the definition of $C_1$, we get the error estimation \cref{error estimation 1}.
\end{proof}

\begin{theorem}[Projection error estimate]\label{main theorem}
The following error estimate holds:
\begin{equation}
\|u-P_h u\|_N \leq C_h \|u-P_h u\|_M \quad \forall  u \in {V},
\end{equation}
where
\begin{eqnarray}\label{C_h_Definition}
C_h :=  0.6711
        \max_{K\in K_b^h}\frac{h_K}{\sqrt{H_K}}
      + \frac{0.1893}{\sqrt{\lambda_{h,1}}}\max_{K\in K^h} h_K.
\end{eqnarray}
\end{theorem}

\begin{proof}
For all $u\in {V}$, by using \cref{error estimation 2} and \cref{error estimation 1}, we obtain
\begin{eqnarray*}
\|u-P_h u\|_N &\leq& \|u-\Pi_h u\|_N + \|\Pi_h u-P_h u\|_N \nonumber\\
&\leq& 0.6711
       \max_{K\in K_b^h}\frac{h_K}{\sqrt{H_K}}  |u-\Pi_h u|_{1,\Omega} \\
&&   + \frac{0.1893}{\sqrt{\lambda_{h,1}}}\max_{K\in K^h} h_K |u-\Pi_h u|_{1,\Omega} \nonumber\\
&=&  C_h |u-\Pi_h u|_{1,\Omega} \leq C_h |u-P_h u|_{1,\Omega} \leq C_h \|u-P_h u\|_M.
\end{eqnarray*}
The second-last inequality holds because of the orthogonality of $\Pi_h$ in \cref{orthogonality_interpolation}.
\end{proof}

\subsection{Explicit lower eigenvalue bounds}
With \cref{framework_theorem} and the explicit error estimate for $P_h$  in \cref{main theorem},
we can now obtain explicit lower bounds for Steklov eigenvalues.

\begin{theorem}[Explicit lower bounds for Steklov eigenvalues]\label{steklov_explicit_eig_bound}
Let $\lambda_{h,i}$ be the approximate eigenvalues of \cref{steklov fem}.
We have the following lower bounds for eigenvalues of the Steklov eigenvalue problem \cref{steklov problem},
\begin{equation}
\label{lower-bound-conclusion}
\lambda_i \ge  \frac{\lambda_{h,i}}{1+C_h^2 \lambda_{h,i}}, ~~ i=1,2,\ldots, n\:.
\end{equation}
Here, $n=dim({\mbox{Ker} (\gamma_h)^\perp})$ and  $C_h$ is the quantity defined in \cref{C_h_Definition}.
\end{theorem}

\begin{remark}
Note that since $C_h=O(\sqrt{h})$ as $h\to 0$, the lower eigenvalue bound obtained in
\cref{lower-bound-conclusion} only converges at a rate of $O(h)$. This is not
optimal when compared with the approximate eigenvalues themselves, which have a convergence rate of
$O(h^2)$ for solutions with $H^{2}$-regularity; see, e.g., \cite{Yang2010}. An idea to recover the convergence rate is
to utilized the property of constant described in \cref{remark:trace-constant} and refine the mesh for boundary elements.
\end{remark}

\section{Computation Results}

Two example Steklov eigenvalue problems are considered here, one on the unit square $\Omega = (0,1)\times(0,1)$
and the other on the L-shaped domain $\Omega = (0, 2)\times(0,2)\backslash[1,2]\times[1,2]$.
For each example, explicit lower eigenvalue bounds are obtained by applying \cref{steklov_explicit_eig_bound}.

In order to estimate the floating-point rounding errors, interval arithmetic is utilized
for the numerical computation to guarantee that the results are
mathematically correct. The method of Behnke \cite{Behnke-1991} is used, along with the INTLAB toolbox, developed by Rump \cite{Rump-1999},
to give verified eigenvalue bounds for
the generalized matrix eigenvalue problems.

\subsection{Unit square domain}
We uniformly triangulate the domain $\Omega = (0,1)\times(0,1)$;
see \cref{unit square mesh} for a sample mesh with mesh size $h=1/8$.
Here, the mesh size is defined by the length of the triangle side adjacent to the
right angle. Note that the maximum edge length for each element is
$h_K=\sqrt{2}/8$.

\begin{figure}[!ht]
\begin{center}
\includegraphics[width=3.5cm,height=3.5cm,angle=90]{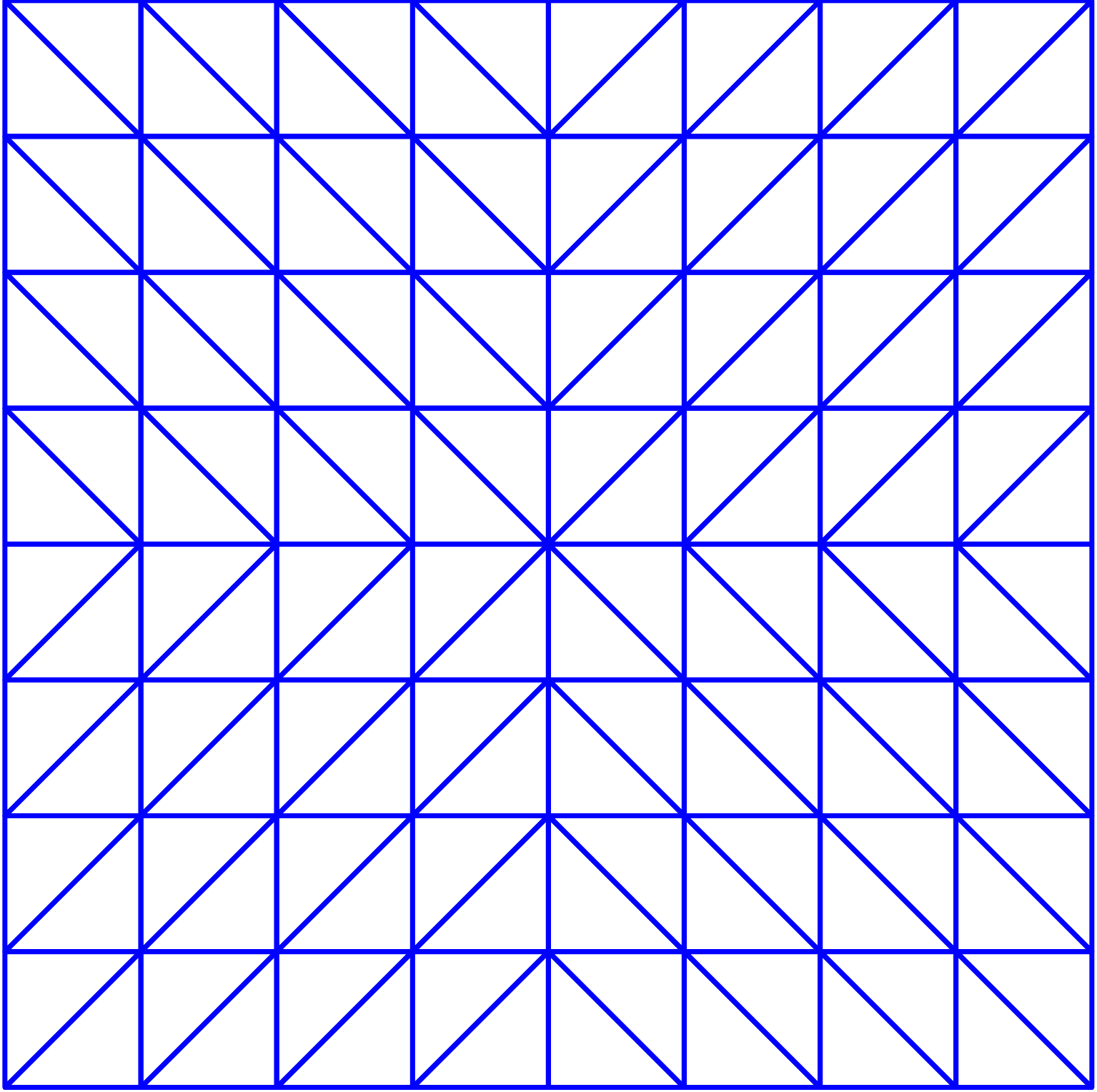}
\caption{Sample uniform triangulation for the unit square ($h=1/8$)}
\label{unit square mesh}
\end{center}
\end{figure}

%\paragraph{Rough eigenvalue bounds}
Over a quite refined mesh, the approximate eigenvalues $\widetilde{\lambda}_i$ ($i=1,2,\ldots,5$) with better precision are calculated; see \cref{table-ref-eig-unit-square}.
However, these results are not guaranteed to be strictly correct.
In \cref{unit-square-CR-CG}, we show verified eigenvalue bounds of the leading $5$ eigenvalues for different mesh sizes.
For example, $(0.231, 0.241)$ in the $1/8$ column and $\lambda_1$ row means that $0.231 < \lambda_1 < 0.241$ in the case $h=1/8$.
The lower bounds are obtained using
\cref{steklov_explicit_eig_bound} together with the Crouzeix--Raviart FEM,
while the upper bounds are obtained using the first-order Lagrange FEM.
%It can be seen that the upper bounds are more accurate than lower bounds.
%More specifically, with mesh size $h = 1/64$, the value of $C_h$ is 0.2327
%and further we can get the lower bound of the leading five eigenvalues at one time.

\begin{table}[tbhp]
\caption{Approximate eigenvalue over refined mesh ($h=1/256$).}
\label{table-ref-eig-unit-square}
\begin{center}
\begin{tabular}{|c|c|c|c|c|}\hline
\rule[-0.02cm]{0cm}{0.40cm}{}
 $\widetilde{\lambda}_1$ &  $\widetilde{\lambda}_2$ &  $\widetilde{\lambda}_3$ &  $\widetilde{\lambda}_4$ &  $\widetilde{\lambda}_5$ \\ \hline
 0.240079 & 1.492293 & 1.492293 &  2.082616  & 4.733516 \\ \hline
\end{tabular}
\end{center}
\end{table}

To investigate the convergence rate of the approximate eigenvalues, we consider
the total errors for the lower and upper eigenvalue bounds, denoted by
$\lambda_{i, \text{lower}}$ and $\lambda_{i, \text{upper}}$, respectively.
%~ High-precision approximate eigenvalues obtained using denser meshes and higher-order FEMs,
%~ denoted by $\widetilde{\lambda}_i$, are displayed in \cref{fig-unit-square-thm-lg}.
The total errors $Err_{\text{lower}}$ and $Err_{\text{upper}}$ are defined by
$$Err_{\text{lower}}:=\sum_{i=1}^{5} |\widetilde\lambda_i - \lambda_{i, \text{lower}}|,
\quad Err_{\text{upper}}:=\sum_{i=1}^{5} |\widetilde\lambda_i - \lambda_{i,\text{upper}}|\:.
$$
The convergence rates of the total errors $Err_{\text{lower}}$ and $Err_{\text{upper}}$
are denoted by $\sigma_{\text{upper}}$ and $\sigma_{\text{lower}}$, respectively,  in \cref{unit-square-CR-CG}.
It can be seen that the upper bound has much better convergence than the lower bound.
% is deteriorated to $O(h)$.

\begin{table}[tbhp]
\caption{Verified eigenvalue bounds for the unit square domain.
(Only 4 significant digits are displayed due to space limitations%
%The upper bounds for $\lambda_1$ and $\lambda_2$ have more correct significant digits,
%which are cut due to the space limitation.
)}
\label{unit-square-CR-CG}
\begin{center}
\begin{tabular}{c|c|c|c|c}\hline
$h$                     & 1/8            & 1/16           & 1/32           & 1/64         \\ \hline
$C_h$                   & 0.4039         & 0.2715         & 0.1849         & 0.1272       \\ \hline
$\lambda_{1}$         & (0.231, 0.241) & (0.235, 0.241) & (0.238, 0.241) & (0.239, 0.241) \\ \hline
$\lambda_{2}$         & (1.195, 1.503) & (1.342, 1.496) & (1.419, 1.494) & (1.456, 1.493) \\ \hline
$\lambda_{3}$         & (1.195, 1.503) & (1.342, 1.496) & (1.419, 1.494) & (1.456, 1.493) \\ \hline
$\lambda_{4}$         & (1.541, 2.148) & (1.800, 2.099) & (1.942, 2.087) & (2.014, 2.084) \\ \hline
$\lambda_{5}$         & (2.570, 4.897) & (3.456, 4.779) & (4.054, 4.746) & (4.391, 4.737) \\ \hline
$\sigma_{\text{lower}}$ &       -        &  0.83          &  0.95          &  1.00        \\ \hline
$\sigma_{\text{upper}}$ &       -        &  1.90          &  1.97          &  1.99        \\ \hline
\end{tabular}
\end{center}
\end{table}

\subsection{Domain with re-entrant corner (L-shaped)}

Here, we consider the eigenvalue bounds for a Steklov eigenvalue problem on the L-shaped domain $\Omega = (0, 2)\times(0,2)\backslash[1,2]\times[1,2]$.
Non-uniform meshes are used in the FEM computations.
\cref{mesh for lshape} shows a sample non-uniform mesh with a geometrically graded
triangular subdivision, where $h_K=O(r^{1/3})$ and $r$ is the distance from the element $K$ to the corner.

\begin{figure}[!ht]
\begin{center}
\includegraphics[width=4cm,height=4cm,angle=-90]{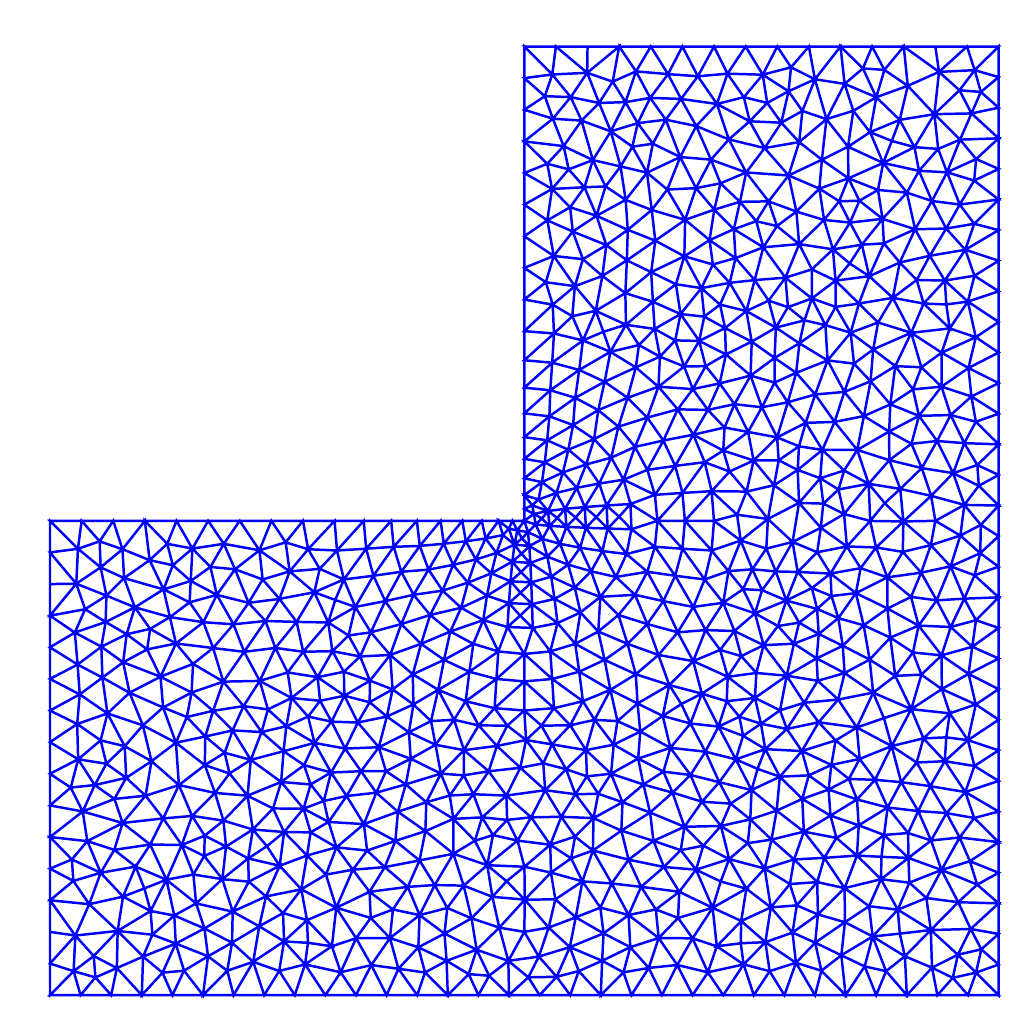}
\end{center}
\caption{Non-uniform mesh for the L-shaped domain}
\label{mesh for lshape}
\end{figure}

Mathematically rigorous lower and upper bounds for the leading 5 eigenvalues are listed in \cref{lshape thm 3.9 delaunay}.
The lower bounds are obtained using \cref{steklov_explicit_eig_bound} together with $4916$ elements and $C_h = 0.2224$. %4916 elements 0.2224
The upper bounds are acquired using a linear Lagrange finite element space.
\begin{table}[!ht]
\caption{Verified eigenvalue bounds for the L-shaped domain.}
\label{lshape thm 3.9 delaunay}
\begin{center}
\begin{tabular}{c|c|c|c|c}\hline
\rule[-0.02cm]{0cm}{0.50cm}{}
$i$  & Lower Bound   &  CR Element ($\lambda_{h,i}$)   &  $\widetilde{\lambda}_i$    & Upper Bound  \\ \hline
1    & 0.33575       &   0.34141     &     0.34141         & 0.34143      \\ \hline
2    & 0.59833       &   0.61673     &     0.61686         & 0.61717      \\ \hline
3    & 0.93844       &   0.98421     &     0.98427         & 0.98448      \\ \hline
4    & 1.56047       &   1.69159     &     1.69206         & 1.69332      \\ \hline
5    & 1.56791       &   1.70041     &     1.70092         & 1.70230      \\ \hline
\end{tabular}
\end{center}
\end{table}

\section{Summary}
In this paper, we propose an abstract framework that provides computable lower eigenvalue
bounds for variationally formulated eigenvalue problems. The framework is successfully applied to
the Steklov eigenvalue problem and explicit lower eigenvalue bounds are obtained
in conjunction with the Crouzeix--Raviart FEM.
Due to the error term related to the trace theorem in \cref{main theorem},
the guaranteed lower bound obtained in \cref{steklov_explicit_eig_bound} cannot achieve an optimal convergence rate even for convex domain, compared with
the asymptotic theoretical analysis.
As a future work, we will try to apply the Lehmann--Goerisch method to improve the convergence rate of rigorous eigenvalue bounds. \\

{\bf Acknowledgement} The authors show grateful thanks to M. Plum of Karlsruhe Institute of Technology for valuable comments at early stage of paper preparation.

\bibliography{library}

\end{document}